\documentclass[11pt]{article}
\usepackage{ragged2e}
\usepackage[dvipsnames]{xcolor}
\usepackage{amsfonts,latexsym,amssymb,amsthm,amsmath,graphicx,cases,comment}
\usepackage{amsmath}
\usepackage{paralist}
\usepackage{graphics} 
\usepackage{epsfig} 
\usepackage{mathrsfs}
\usepackage{epstopdf}
\usepackage{lipsum}
\usepackage{newpxtext,newpxmath}

\usepackage[titletoc,title]{appendix}
\usepackage{empheq}
\usepackage{cancel}
\usepackage{tikz}
\usetikzlibrary{arrows,shapes}
\usetikzlibrary{decorations.pathmorphing,decorations.pathreplacing}
\usetikzlibrary{calc,patterns,angles,quotes}
\usepackage{dsfont}
\usepackage{cases}

\usepackage{float}

\usepackage[colorlinks=true]{hyperref}
\hypersetup{urlcolor=blue, citecolor=blue}

\newtheorem{theorem}{Theorem}[section]
\newtheorem{lemma}[theorem]{Lemma}

\newtheorem{ass}[theorem]{Assumption}
\newtheorem{prop}[theorem]{Proposition}
\newtheorem{cor}[theorem]{Corollary}

\theoremstyle{definition}

\newtheorem{rmk}[theorem]{Remark}

\newcommand{\D}{\mathcal{D}}
\newcommand{\io}{\int_\Omega}
\newcommand{\ig}{\int_\Gamma}

\newcommand{\dv}[1]{\mbox{div}(#1)}

\newcommand{\n}[1]{\mathds{#1}}
\newcommand{\p}[1]{A^{1/2}{#1}}

\allowdisplaybreaks
\newcommand{\thistheoremname}{}

\newtheorem*{genericthm*}{\thistheoremname}
\newenvironment{namedthm*}[1]
{\renewcommand{\thistheoremname}{#1}%
	\begin{genericthm*}}
	{\end{genericthm*}}

\begin{document}
\title{Boundary Stabilization of the linear MGT equation with Feedback Neumann control}
	
	\author{
		 Marcelo Bongarti and Irena Lasiecka\\
		{\small Department of Mathematical Sciences}\\
		{\small University of Memphis}\\
		{\small Memphis, TN 38152 USA}\\
	}
	\maketitle
	\begin{abstract}
The Jordan--Moore--Gibson--Thompson (JMGT)\cite{christov_heat_2005,jordan_nonlinear_2008,straughan_heat_2014}  equation  is a benchmark model describing  propagation of nonlinear acoustic waves in heterogeneous fluids at rest. This is  a third order (in time) dynamics which accounts for a finite speed of propagation of heat signals (see \cite{coulouvrat_equations_1992,crighton_model_1979,jordan_nonlinear_2008,jordan_second-sound_2014,kaltenbacher_jordan-moore-gibson-thompson_2019}).  In this paper, we study a  boundary stabilization of  linearized version (also known as MGT--equation) in the {\it  critical case}, configuration in which the smallness of the diffusion effects leads to conservative dynamics \cite{kaltenbacher_wellposedness_2011}. Through a single measurement in {\it{feedback}} form made on a non--nempty, relatively open portion of the boundary under natural geometric conditions, we were able to obtain uniform exponential stability results that are, in addition, uniform with respect to the space--dependent viscoelasticity parameter which no longer needs to be assumed positive and in fact can be degenerate and  taken to be zero on the whole domain.This result, of independent interest in the area of boundary stabilization of MGT equations, provides a necessary first step for the study of optimal boundary feedback control on {\it infinite horizon} \cite{bucci_feedback_2019}.
	\end{abstract}
	\noindent{\bf keywords: } Nonlinear acoustics, second sound, third order in time,MGT equation, heat-conduction, boundary stabilization, degenerate viscoelasticity.
\section{Introduction}
Let $\Omega \in \mathbb{R}^3 $ bounded with $C^2$--boundary denoted by $\Gamma$. 
With $T>0$ (could also be $T = \infty$), the third-order (in time) quasilinear JMGT--equation is given by \begin{equation}\label{eqnl}\tau u_{ttt} + (\alpha-2ku)u_{tt} - c^2\Delta u - b \Delta u_t = 2ku_t^2 \ \ \mbox{in}  \ \ (0,T) \times \Omega \end{equation}
where $k>0$ is a nonlinearity parameter, $c>0$ denotes the speed of sound, $\tau > 0$ denotes thermal relaxation parameter, $b := \delta + \tau c^2 > 0$ where $\delta$ denotes the sound diffusivity and $\alpha$ refers to viscoelastic friction. The parameter  $\tau > 0$  -- introduced when the Fourier's law of heat conduction is replaced my the Maxwell--Catanneo's (MC) law -- plays an important role: it removes the paradox of infinite speed of propagation in heat waves. 

 A multitude of applications ranging from acoustics, image processing, thermodynamics, etc. have brought a  considerable  attention to the dynamics behind the hyperbolic (third--order in time) acoustic wave models. As a  consequence, a rich literature on the topic has been developed  during  the last decade \cite{alves_mooregibsonthompson_2018,
 bongarti_singular_2020,
 bongarti_vanishing_2020,
 boulaaras_galerkin_2019,
 bucci_cauchy-dirichlet_2020,
 bucci_regularity_2020,
 cavalcanti_long_2016,
 chen_nonexistence_2020,
 conejero_chaotic_2015,
 delloro_mooregibsonthompson_2016,
 delloro_note_2019,
 delloro_mooregibsonthompson_2017,
 kaltenbacher_mathematics_2015,
 lasiecka_mooregibsonthompson_2015,
 lasiecka_mooregibsonthompson_2016,
 liu_inverse_2014,
 liu_general_2020,
 nikolic_mathematical_2020,
 nikolic_jordan-moore-gibson-thompson_2020,
 pellicer_cauchy_2019,
 pellicer_optimal_2019,
 racke_global_2020}. The interest  behind the propagation of waves through viscous fluids and other heterogeneous media, has been already pointed out -- almost two centuries ago -- by Professor Stokes in his prominent article \cite{stokes_examination_1851} in 1851. Stokes' work was rooted on the idea that heat propagation -- in particular within acoustic media -- was hyperbolic. Later experimental studies indicated the presence of heat {\it waves} in such materials thereby dictating the presence of the (nowadays known as) thermal relaxation parameter $\tau > 0$.
 
 There is a strong connection between Stokes--model and the JMGT model, both  revealing the basic principles of acoustic waves propagation in the presence of  the heat waves. 
 This also led to  a creation  of other  models accounting for a more detailed information 
 about the material and the medium \cite{jordan_second-sound_2014,christov_heat_2005,straughan_heat_2014}. In  these models, 
 standard time derivative may be replaced by material derivative, which depends on the medium. Dependence of the heat flux upon the media, its surroundings and other thermal--material dependent quantities: thermal inertia, specific heat, velocity field, etc. are taken into account in the respective  modeling processes. 
  For a broader understanding of hyperbolic heat and modern nonlinear acoustics, we refer to  the works \cite{cattaneo_form_1958,cattaneo_sulla_2011,jordan_nonlinear_2008,jordan_second-sound_2014,lai_introduction_2010,straughan_heat_2014} and references therein. Recent review of pertinent  modeling aspects of  acoustic waves can also  be found in \cite{jordan_nonlinear_2008,kaltenbacher_mathematics_2015,m_jordan_introduction_2019}.
  In order to focus our work on boundary stabilization and the related technical details, 
  in the present work we consider $\tau > 0$  and fixed. However, other generalizations may be possible and, indeed, welcome. 

Typical boundary conditions associated with the model \eqref{eqnl}  and its linearization are homogenous Dirichlet $ u =0 $ imposed  on $\Gamma$.  Questions such as wellposedness of solutions and their stability were extensively studied \cite{kaltenbacher_exponential_2012,kaltenbacher_wellposedness_2011,kaltenbacher_wellposedness_2012,marchand_abstract_2012}.
With $k =0$ equation \eqref{eqnl} is linear and reads \begin{equation}\label{eqnlconst}\tau u_{ttt} + \alpha u_{tt} - c^2\Delta u - b \Delta u_t = 0 \ \ \mbox{in}  \ \ (0,T) \times \Omega \end{equation} (here we allow $\alpha \in \mathbb{R}$). The structurally damped case ($b > 0$) corresponds to a  group generator defined on the phase space
$H_0^1(\Omega) \times H_0^1(\Omega) \times L^2(\Omega)$.  In the absence of structural damping ($b =0$), however, semigroup generation fails and the problem is ill-posed, a property dating back to Fattorini \cite{fattorini_ordinary_1969}
in 1969. Nonlinear semigroups corresponding to \eqref{eqnl}  were  shown to exist in the  following phase spaces $H^2(\Omega) \cap H_0^1(\Omega) \times H_0^1(\Omega) \times L^2(\Omega)$, $~~[H^2(\Omega) \cap H_0^1(\Omega)]^2 \times H_0^1(\Omega)$.  First with the data assumed to be small at the  level  of the underlying phase space \cite{kaltenbacher_wellposedness_2012} and later with smallness required only in the lower topology, namely $H_0^1(\Omega) \times H_0^1(\Omega) \times L^2(\Omega)$ \cite{bongarti_singular_2020}.

The parameter $\gamma:= \alpha -  \dfrac{\tau 
c^2}{b}$  is critical for stability of solutions: both linear  and nonlinear semigroups are exponentially stable provided that $ \gamma>0$. When $\gamma =0$ the linear dynamics is conservative and therefore there is no decay of solutions. This brings us to the goal of the present paper, namely the {\it linear} stabilization: how to stabilize the model with critical $\gamma$-i.e $\gamma$ can be degenerate?  It is known by now that specifically constructed memory terms  may have stabilizing effects on the critical dynamics \cite{delloro_mooregibsonthompson_2016,delloro_mooregibsonthompson_2017,lasiecka_mooregibsonthompson_2016}. In this work we are  interested how to achieve  stabilization in the critical case  via a  boundary feedback only.It is known, that boundary of the region is accessible to external manipulations, hence a good place for placing actuators and sensors. Within this spirit, we  shall show that a suitable boundary feedback implemented on the boundary $\Gamma $ will lead to exponential stability of the  resulting semigroup. This result is important not only  on its own rights, but also within the context of the quasilinear  equations where  questions of stability are strongly linked to  decay properties  of linearized solutions \cite{kaltenbacher_wellposedness_2012}. 

 \subsection{The linearized PDE model with space--dependent viscoelasticity}
 
 Let the domain $\Omega \subset \n{R}^n (n = 2,3)$ with the  boundary $\Gamma = \partial \Omega$ be of class $C^2$. Let $\Gamma$  be  divided into two disjoint parts, $\Gamma_0$ and $\Gamma_1$, both non-empty, with $\Gamma_1$ relatively open in $\Gamma.$ We consider the linear version of \eqref{eqnl} but  with a space--dependent coeficient $\alpha \in C(\overline{\Omega})$ 
 \begin{equation}\label{eqnl1}\tau u_{ttt} + \alpha(x) u_{tt} - c^2\Delta u - b \Delta u_t = 0 \ \ \mbox{in}  \ \ (0,T) \times \Omega \end{equation}
 and subject to the Robin-\emph{type} (on $\Gamma_0$) and Dirichlet (on $\Gamma_1$) boundary conditions \begin{equation} \label{BC}
\dfrac{\partial u}{\partial \nu} +\eta u_t = 0 \ \ \mbox{on} \ \ (0,T) \times \Gamma_0 \ \ \mbox{and} \ \ u= 0 \ \ \mbox{on} \ \ (0,T) \times \Gamma_1
\end{equation} where $\eta > 0$, and the initial conditions are  given by \begin{equation} \label{IC}
u(0)=u_0,u_t(0)=u_1, u_{tt}(0) = u_2.
\end{equation}
In addition to continuity, we assume that the damping coefficient $\alpha(x) > 0$ is such that stability parameter $\gamma \in C(\overline{\Omega})$ satisfies
\begin{equation}\label{gamma}
\gamma(x):=\alpha (x) -\frac{\tau c^2}{b} \geqslant 0.
\end{equation}
\begin{rmk}
The above assumption will be used only for the stability estimates. The generation of a semigroup  is valid without assuming (\ref{gamma}).
\end{rmk}Since the damping typically depends on local properties of the material, assuming variability of the damping  $\alpha(x)$ is physically relevant and, in most cases, necessary. 
We shall show that under suitable geometric conditions imposed on $\Omega$ the linear system is exponentially  stable in the topology of the natural phase space.

\

\noindent{\bf Notation:} Throughout this paper, $L^2(\Omega)$ denotes the space of Lebesgue measurable functions whose squares are integrable and $H^s(\Omega)$ denotes the $L^2(\Omega)$-based Sobolev space of order $s$. Moreover, we use the notation $H_{\Gamma}^1(\Omega)$ to represent the space \begin{equation}
H_\Gamma^1(\Omega) := \left\{u \in H^1(\Omega); u\rvert_\Gamma = 0\right\}
\end{equation} instead of the standard $H_0^1(\Omega)$ in order to emphasize the portion of the boundary on what the trace is vanishing . We also denote by $H_\Gamma^2(\Omega)$ the space $H^2(\Omega) \cap H_{\Gamma}^1(\Omega).$ We denote the inner product in $L^2(\Omega)$ and $L^2(\Gamma)$ respectively by $$(u,v) = \int_\Omega uvd\Omega \ \ \mbox{and} \ \ (u,v)_\Gamma = \int_\Gamma uvd\Gamma$$ and the respective induced norms in $L^2(\Omega)$ and $L^2(\Gamma)$ are denoted by $\|\cdot\|_2$ and $\|\cdot\|_\Gamma$ respectively. 

\subsection{Main results and discussion}
We begin with the abstract version of equation \eqref{eqnl1}.To this end, let $A: \D(A) \subset L^2(\Omega) \to L^2(\Omega)$ be the operator defined as \begin{equation}\label{oplap} A\xi = -\Delta \xi,  \ \ 
\D(A) = \left\{\xi \in H^2(\Omega); \  \xi\biggr\rvert_{\Gamma_1} = \dfrac{\partial \xi}{\partial \nu} \biggr \rvert_{\Gamma_0} \equiv 0 \right\}.
\end{equation} It is well known that $A$ is a positive ($\Gamma_1 \ne 0$),  self-adjoint operator with compact resolvent and that $\D\left(\p\ \right) = H_{\Gamma_1}^1(\Omega)$ (equivalent norms). In addition, up to a bit of abuse of notation we denote (also) by $A: L^2(\Omega) \to [\D(A)]'$ the extension (by duality) of the operator $A: \D(A) \subset L^2(\Omega) \to L^2(\Omega)$ defined in \eqref{oplap}.

Next, for $\varphi \in L^2(\Gamma_0)$, let $\psi = N(\varphi)$ be the unique solution of the elliptic problem \begin{equation*}
\begin{cases}
\Delta\psi = 0 \ & \mbox{in} \ \Omega \\ \dfrac{\partial \psi}{\partial \nu} = \varphi \ & \mbox{on} \ \Gamma_0 \\ \psi = 0 \ & \mbox{on} \ \Gamma_1.
\end{cases}
\end{equation*} It follows from elliptic theory that $N \in \mathcal{L}(H^s(\Gamma_0),H^{s+3/2}(\Omega))$\footnote{$\mathcal{L}(X,Y)$ denote the space of linear bounded operators from $X$ to $Y$} $(s \in \mathbb{R}_+^\ast)$ and \begin{equation} \label{neq} N^\ast A \xi = \begin{cases}
\xi \ \mbox{on} \ &\Gamma_0 \\ 0 \ \mbox{on} \ &\Gamma_1
\end{cases}\end{equation} for all $\xi \in \D(A)$, where $N^\ast$ represents the adjoint of $N$
when the latter is considered as an operator from $L_2(\Gamma_0) $ to $L_2(\Omega) $.

The introduction of $A$ and $N$ will allow us to write equation \eqref{eqnl1} abstractly as \begin{align}\label{absver}
\tau u_{ttt} + \alpha(x) u_{tt} + c^2A(u + \eta NN^\ast A u_t) + bA(u_t + \eta \eta NN^\ast Au_{tt}) = f.
\end{align}

The abstract version of our model gives rise to the natural phase space we are going to consider. We define $\mathbb{H}$ as \begin{equation}\label{ph-sp}\mathbb{H} := H_{\Gamma_1}^1(\Omega) \times H_{\Gamma_1}^1(\Omega) \times L^2(\Omega)\end{equation}

The computations leading to \eqref{absver} show formally that $u$ is a solution of \eqref{eqnl1} with boundary condition given by \eqref{BC} if and only if $\Phi = (u,u_t,u_{tt})^\top$ is a solution for the first other system \begin{equation} \begin{cases}
\label{usist} \Phi_t = \mathscr{A}\Phi + F\\ 
\Phi(0) = \Phi_0 = (u_0,u_1,u_2)^\top,\end{cases}
\end{equation} with $\mathscr{A}:\D(\mathscr{A}) \subset \mathbb{H} \to \mathbb{H}$ given by\begin{equation}\label{opuA}
\mathscr{A}(\xi_1,\xi_2,\xi_3)^\top := (\xi_2,\xi_3,-\alpha\tau^{-1}\xi_3-c^2\tau^{-1} A(\xi_1+ \eta NN^\ast A\xi_2) -b\tau^{-1}A(\xi_2 + \eta NN^\ast A\xi_3))
\end{equation} where {\footnotesize \begin{align}
\label{domA} \D(\mathscr{A}) &= \left\{(\xi_1,\xi_2,\xi_3)^\top \in \mathbb{H};    \ \xi_3 \in \D\left(\p\ \right), \ \xi_1 + \eta NN^\ast A\xi_2 \in \D(A), \ \xi_2 + \eta NN^\ast A\xi_3 \in \D(A)\right\} \nonumber \\ &= \left\{(\xi_1,\xi_2,\xi_3)^\top \in \left[H_{\Gamma_1}^2(\Omega)\right]^2 \times \D \left(\p\ \right); \ \left[\dfrac{\partial \xi_1}{\partial \nu} + \eta\xi_2\right] \biggr\rvert_{\Gamma_0}= 
	\left[\dfrac{\partial \xi_2}{\partial \nu} + \eta\xi_3\right]\biggr\rvert_{\Gamma_0} =0
\right\}
\end{align}} and $F^\top =(0,0,f).$

We now recall that, topologically, the space $\mathbb{H}$ is equivalent to 
\begin{equation}\label{lowen}\D\left(\p\ \right) \times \D\left(\p\ \right) \times L^2(\Omega)\end{equation} with the topology induced by the inner product \begin{equation}
\label{inH} \left((\xi_1,\xi_2,\xi_3)^\top,(\varphi_1, \varphi_2, \varphi_3)\right)_{\mathbb{H}} = (\p \xi_1,\p \varphi_1) + b(\p\xi_2,\p \varphi_2) + (\xi_3,\varphi_3),
\end{equation} for all $(\xi_1,\xi_2,\xi_3)^\top, (\varphi_1,\varphi_2,\varphi_3)^\top \in \mathbb{H}.$ Because of this equivalence we will be using the same $\mathbb{H}$ to denote both spaces.

We are then in position to state  our wellposedness of the phase space solutions. 
\begin{theorem}\label{t1}
Assume $f \in L^1((0,T),L^2(\Omega))$, and $\eta\geq 0 $, $\alpha \in C(\Omega)$ . For every initial data $\Phi_0 :=(u_0,u_1,u_2)$ in $\mathbb{H}$,
there exists a unique semigroup solution $\Phi =(u,u_t,u_{tt}) $ such that 
$\Phi \in C([0,T], \mathbb{H}) $ for every $T > 0.$ Moreover, if the initial datum belongs to $\D(\mathscr{A})$ and $f \in C^1([0,T],L^2(\Omega))$ the corresponding solution is in $C((0,T]; D(\mathscr{A})) \cap C^1([0,T], \mathbb{H})$.
\end{theorem}

Our main results is exponential stability of the solutions refereed to in Theorem \ref{t1}.
In order to formulate the result we need to impose geometric condition:
\begin{ass}\label{asgeo}
\begin{enumerate}
	\item we choose a point $x_0 \in \n{R}^n$ \emph{outside} of $\overline{\Omega}.$
	and  we define the vector field $h: \mathds{R}^n \to \n{R}^n $ given by $h(x) = x - x_0$.
  With $\nu(x)$ denoting the outwards normal unit vector at $x$ we define $\Gamma_0,\Gamma_1 \subseteq \Gamma$ by $$\Gamma_0 = \{x \in \Gamma; \ \nu(x)\cdot h(x) > 0\},~~~~\Gamma_1 = \{x \in \Gamma; \ \nu(x)\cdot h(x) \leqslant 0\}.$$ 
\end{enumerate} 
\end{ass}
\begin{rmk}
The geometric condition imposed above can be substantially relaxed. For instance, if  the feedback control is active on the full $\Gamma$, there is no need for any geometric coonditions. However, this would  require some microlocal analysis and becomes rather technical  \cite{lasiecka_uniform_1993-1,lasiecka_uniform_1992,tataru_regularity_1998}. For this reason we opted for a more restrictive geometry, as to make the exposition fully independent. \end{rmk}
\begin{theorem}\label{t2}
Let Assumption \ref{asgeo} and condition \eqref{gamma} be satisfied. Then, there exist $\omega > 0$ and $M>0$ such that \begin{equation}\label{exp} \|\Phi(t)\|_{\mathbb{H}} \leqslant Me^{-\omega t}\|\Phi_0\|_\mathbb{H}\end{equation} for all $t \geqslant 0.$
\end{theorem}

   As mentioned earlier, the {\it wellposedness} of the MGT equation with homogenous boundary conditions is well known by now \cite{kaltenbacher_wellposedness_2011,kaltenbacher_wellposedness_2012,marchand_abstract_2012}. 
    However, in the  case of non-homogeneous boundary conditions, the situation is much more complicated due to the fact that "wave" operator with Neumann boundary data does not satisfy the Lopatinski condition \cite{sakamoto_hyperbolic_1982,tataru_regularity_1998}, unless $\Omega$ is one dimensional. This leads to a {\it loss} of  $1/3$
derivative when looking at the map from the boundary with $L_2$ data into the $H^1\times L^2$ 
solutions. This has been known for some time in the case of wave equation, but only recently 
studied for MGT equation \cite{bucci_cauchy-dirichlet_2020,bucci_regularity_2020,triggiani_sharp_2020}. In fact, an open loop control problem for MGT equation with $L^2$--Neumann  controls leads to only distributional solutions \cite{bucci_feedback_2019}. 
However, for both the wave and MGT equations, a {\it boundary feedback via Neumann boundary conditions does recover this lo}  leading to the  recuperation of the full energy. This is due to the  boundary  dissipativity with  $\eta \geq 0$. 
 The main mathematical issue in dealing with the  nonhomogenous Neumann boundary data is  to deal  with {\it unbounded and uncloseable} perturbations within the context of the third order hyperbolic dynamics.  In fact, this is the  first result on wellposedness of feedback generator  for  MGT dynamics. 
Also, notice that applying feedback $\frac{\partial u}{\partial \nu} +\eta u_t=0$, with $\eta <0$  leads to the  ill-posed dynamics.
This corresponds to anti--damping which shifts the spectrum to a "wrong" complex half--plane-thus denying wellposedness of a semigroup. While  the analog of  stability result in Theorem \ref{t2} is known for the case of wave equation, this result  is new for the MGT ($\tau > 0$) with critical stability parameter  $\gamma$. The difficulties encountered in the proof of wellposedness in Theorem \ref{t1} are compounded, when proving Theorem \ref{t2}, by geometric considerations necessary when studying dissipation with restricted geometric support (such as  portion of the boundary).Geometric condition assumed in  Assumption (\ref{asgeo})  can be, however,  substantially relaxed. For instance, when the dissipation is active on the full boundary there is no need for any geometric constraints \cite{lasiecka_uniform_1993-1,lasiecka_uniform_1992}. However, this brings forward microlocal analysis arguments which are known by now, however tedious. In order to ease readability and focus of the analysis, we have opted for a more restrictive version of geometric assumption.  

    It should be noted that the result of Theorem \ref{t2} is critical when studying optimal boundary feedback control problem for MGT equation. While the feedback synthesis for this model has been carried out in \cite{bucci_feedback_2019}  for a finite horizon problem, analysis of {\it infinite horizon} problem requires stabilizability condition, which is provided by Theorem \ref{t2}.

\section{Wellposedness: proof of Theorem \ref{t1}}

The thermal relaxation parameter $\tau > 0$ plays no significant role in the study of the wellposedness of \eqref{eqnl1}, therefore for the sake of readability we are assuming $\tau = 1$ in this section. 

The main goal of this section is to prove that the operator $\mathscr{A}: \D(\mathscr{A}) \subset \mathbb{H} \to \mathbb{H}$ generates a strongly continuous semigroup.
It is convenient to introduce the following  change of variables $bz = bu_t + c^2 u$ (see \cite{marchand_abstract_2012}) which reduces the problem to a PDE- abstract ODE coupled system.

Let $M \in \mathcal{L}(\mathbb{H})$ defined by $$M(\xi_1,\xi_2,\xi_3)^\top = \left(\xi_1,\xi_2 + \dfrac{c^2}{b}\xi_1, \xi_3 + \dfrac{c^2}{b}\xi_2\right)$$ which has inverse $M^{-1}\in \mathcal{L}(\mathbb{H})$ given by $$M^{-1}(\xi_1,\xi_2,\xi_3)^\top = \left(\xi_1,\xi_2-\dfrac{c^2}{b}\xi_1,\xi_3 - \dfrac{c^2}{b}\xi_2 + \dfrac{c^4}{b^2}\xi_1\right)$$ and therefore is an isomorphism of $\mathbb{H}$. The next lemma makes precise the translation of \eqref{eqnl1} to the system involving $z.$\begin{lemma} \label{equiv_prob}  Assume that the compatibility condition \begin{equation}\label{comp} \dfrac{\partial}{\partial \nu} u_0 + \eta u_1 = 0 \ \emph{\mbox{on}} \ \Gamma_0
	\end{equation} holds.
	Then $\Phi \in C^1(0,T;\mathbb{H})\cap C(0,T;\D(\mathscr{A}))$ is a strong solution for \eqref{usist} if, and only if, $\Psi= M\Phi \in C^1(0,T;\mathbb{H})\cap C(0,T;\D(\mathbb{A}))$ is a strong solution for  \begin{equation} \begin{cases}
	\label{usistz} \Psi_t = \mathbb{A}\Psi + G\\ 
	\Psi(0) = \Psi_0 = M\Phi_0 =  \left(u_0,u_1 + \dfrac{c^2}{b}u_0,u_2 + \dfrac{c^2}{b}u_1\right)^\top,\end{cases}
	\end{equation} where $G = MF$ and $\mathbb{A} = M\mathscr{A}M^{-1}$ with \begin{equation}\label{domz}
	\D(\mathbb{A}) = \left\{(\xi_1,\xi_2,\xi_3)^\top \in \left[H_{\Gamma_1}^2(\Omega)\right]^2 \times \D \left(\p\ \right);
	\left[\dfrac{\partial \xi_2}{\partial \nu} + \eta\xi_3\right]\biggr\rvert_{\Gamma_0} =0
	\right\}
	\end{equation}
\end{lemma} \begin{proof}
It is simple to check that if $\Phi \in C^1(0,T;\mathbb{H})\cap C(0,T;\D(\mathscr{A}))$ is a strong solution for \eqref{usist} then $\Psi = M\Phi$ belongs to $C^1(0,T;\mathbb{H})\cap C(0,T;\D(\mathbb{A}))$ and satisfy \eqref{usistz}. 

For the reverse implication, the only non-trivial step is to prove that boundary conditions match. To this end, assume that $\Psi = (u,z,z_t) \in C^1(0,T;\mathbb{H}) \cap C(0,T;\D(\mathbb{A}))$ is a strong solution for \eqref{usistz}. Let 
$$\Upsilon(t) := \left(\dfrac{\partial u(t)}{\partial \nu} + \eta u_t(t)\right)\biggr\rvert_{\Gamma_0}, \ t \geqslant 0$$ and notice that $b\Upsilon_t +c^2\Upsilon = 0$ for all $t$. This along with the compatibility condition ($\Upsilon(0) = 0$) implies that $\Upsilon \equiv 0$ hence $(u,u_t,u_{tt}) \in \D(\mathscr{A})$ for all $t$. The proof is then complete.
\end{proof}

\

For $(\xi_1,\xi_2,\xi_3)^\top \in \D(\mathbb{A})$ a basic algebraic computation yields the explicit formula for $\mathbb{A}.$ \begin{equation}
\mathbb{A}(\xi_1,\xi_2,\xi_3)^\top = \left(\xi_2-\dfrac{c^2}{b}\xi_1,\xi_3,-\gamma\left(\xi_3 - \dfrac{c^2}{b}\xi_2 + \dfrac{c^4}{b^2}\xi_1\right) -bA\xi_2-b\eta ANN^\ast A\xi_3\right).
\end{equation}

\ 

We are ready for our generation result.

\begin{theorem}\label{gener}
	The operator $\mathscr{A}$ generates a strongly continuous semigroup on $\mathbb{H}$.
\end{theorem} \begin{proof}
Equivalently, we show that $\mathbb{A}$ generates a strongly continuous semigroup on $\mathbb{H}$. If $\{S(t)\}_{t \geqslant 0}$ is the said semigroup then $\{T(t)\}_{t\geqslant 0}$ ($T(t) := M^{-1}S(t)M, t\geqslant 0$) will be the semigroup generated by $\mathscr{A}.$

Write $\mathbb{A} = \mathbb{A}_d + P$ where $$P(\xi_1,\xi_2,\xi_3) = \left(\xi_2,0,\dfrac{\gamma c^2}{b}\left(\xi_2- \dfrac{c^2}{b}\xi_1\right)+(1-\gamma)\xi_3\right), \ \ (\xi_1,\xi_2,\xi_3)^\top \in \mathbb{H}$$ is bounded in $\mathbb{H}$ and \begin{equation}\mathbb{A}_d(\xi_1,\xi_2,\xi_3) = \left(- \dfrac{c^2}{b}\xi_1,\xi_3,-\xi_3 - bA\xi_2-b\eta ANN^\ast A\xi_3\right), \ (\xi_1,\xi_2,\xi_3)^\top \in \D(\mathbb{A}_d),\end{equation} where $\D(\mathbb{A}_d) := \D(\mathbb{A}).$ It then suffices to prove generation of $\mathbb{A}_d$ on $\mathbb{H}$.

We start by showing dissipativity: for $(\xi_1,\xi_2,\xi_3)^\top \in \D(\mathbb{A})$ we have \begin{align*}
\left(\mathbb{A}_d(\xi_1,\xi_2,\xi_3),(\xi_1,\xi_2,\xi_3)\right)_{\mathbb{H}} &= - \dfrac{c^2}{b}\|\p \xi_1\|_2^2 + b\left(\p \xi_3, \p\xi_2\right) -\|\xi_3\|_2^2 - b(\p\xi_2,\p\xi_3) - b\eta\|\xi_3\|_{\Gamma_0}^2 \\ &= - \dfrac{c^2}{b}\|\p \xi_1\|_2^2 -\|\xi_3\|_2^2 - b\eta\|\xi_3\|_{\Gamma_0}^2\leqslant 0,
\end{align*} hence, $\mathbb{A}_d$ is dissipative in $\mathbb{H}$.

For maximality in $\mathbb{H}$, given any $L = (f,g,h) \in \mathbb{H}$ we need to show that there exists $\Psi = (\xi_1,\xi_2,\xi_3)^\top \in \D(\mathbb{A})$ such that $(\lambda - \mathbb{A}_d)\Psi = L$, for some $\lambda >0.$ This leads to the system of equations: \begin{equation}
\begin{cases}
\lambda \xi_1 + \dfrac{c^2}{b}\xi_1 = f, \\ \lambda \xi_2 - \xi_3 = g, \\
\lambda \xi_3 +\xi_3 + b A[\xi_2 + \eta NN^\ast A\xi_3]= h,
\end{cases}
\end{equation} which implies $\xi_1 = \left(\lambda + \dfrac{c^2}{b}\right)^{-1}f \in \D(\p).$ Moreover, since $A^{-1} \in \mathcal{L}({L^2(\Omega)})$ the third equation above yields $$\left[(\lambda^2 + \lambda)A^{-1}+b + b\lambda\eta NN^\ast A\right]\xi_3 = \lambda A^{-1}h - bg$$ and then by the strictly positivity (in $\D(\p)$) of the operator $K_\lambda$ defined as $$K_{\lambda} := (\lambda^2 + \lambda)A^{-1}+b + b\lambda\eta NN^\ast A$$  (in fact recall $(NN^\ast A\xi,\xi)_{\D(\p)} = \|N^\ast A\xi\|_{\Gamma_0}^2$ for all $\xi \in 
\D(\p)$ therefore $K_\lambda^{-1} \in \mathcal{L}(\D(\p))$) we can write  $\xi_3 = K_\lambda^{-1}(\lambda A^{-1} h - bg) \in \D(\p).$ Finally, $\xi_2 = \lambda^{-1}(\xi_3+g) = \lambda^{-1}(K_\lambda^{-1}(\lambda A^{-1} h - bg)+g) \in \D(\p).$

The final step for concluding membership of $(\xi_1,\xi_2,\xi_3)$ in $\D(\mathbb{A})$ follows from \begin{align*}
b\lambda (\xi_2 + \eta NN^\ast A \xi_3) &= - (\lambda^2 + \lambda)A^{-1}\xi_3 + \lambda A^{-1} h \in \D(A).
\end{align*} 

Generation is then achieved.
\end{proof}
Applying standard semigroup arguments \cite{pazy_semigroups_1992}  to the result of Theorem \ref{gener}, we obtain the following corollary, which in turn completed the proof of Theorem \ref{t1}. 


\begin{cor}[Wellposedness and Regularity] \label{wpG} Assume $f \in L^1(\n{R}_+,L^2(\Omega))$ and that condition \eqref{comp} is at force. Denote by $\{T(t)\}_{t\geqslant 0}$ the semigroup given by Theorem \ref{gener}.
	
	\emph{(i)} If $\Phi_0 \in \mathbb{H}$, then the function $\Phi \in C([0,T]; \mathbb{H})$ defined as
	$$\Phi(t) \equiv T(t)\Phi_0 + \int_0^t T(t-\sigma)F(\sigma)d\sigma, \ \ t \in [0, T]$$ is the unique mild solution for \eqref{usist} in $\mathbb{H}$.
	
	\emph{(ii)} If $\Phi_0 \in \D(\mathscr{A})$ and, in addition, $f \in C^1(\n{R}_+, L^2(\Omega)$ then the function $\Phi \in C^1([0,T]; \mathbb{H})\cap C((0,T);\D(\mathscr{A}))$ defined as
	$$\Phi(t) \equiv T(t)\Phi_0 + \int_0^t T(t-\sigma)F(\sigma)d\sigma \ \ t \in [0, T]$$ is the unique classical solution for \eqref{usist} in $\mathbb{H}$.
\end{cor}

\begin{rmk}
	We notice here that condition \eqref{comp} is not essential for well-posedness of weak solutions. However,  it is critical for  the  regularity of solutions which allows to interpret mild (semigroup) solution in a strong form of equation $\eqref{eqnl1}.$
\end{rmk}

\subsection{Stabilization in $\mathbb{H}$: proof of Theorem \ref{t2}}

In order to allow future  tracking stability rates with respect to the parameter $\tau$  we now resume considering $\tau > 0$. Moreover, we will use the notation $a \lesssim b$ to say that $a \leqslant Cb$ where $C$ is a constant possibly depending on the physical parameters of the model ($\tau,c,b > 0$) but independent of space, time and $\gamma \in C(\overline{\Omega}).$

We will introduce several energy functionals which will be used to describe  long time behavior of mild solutions to \eqref{eqnl1}. 
 For a classical solution $(u,u_t,u_{tt})$ of \eqref{usist} in $\mathbb{H}$, we define the corresponding energy by $E(t) = E_0(t) + E_1(t)$ where $E_i:[0,T] \to \mathbb{R}_+$ ($i = 0,1$) are defined by \begin{align}\label{E1} E_1(t) & := \dfrac{b}{2} \left\|\p{ \left(u_t + \dfrac{\tau c^2}{b}u\right)}\right\|_2^2 + \dfrac{\tau}{2}\left\|u_{tt} + \dfrac{c^2}{b}u_t\right\|_2^2 + \dfrac{c^2}{2b}\|\gamma^{1/2}u_t\|_2^2 \nonumber \\ &= \dfrac{b}{2}\|\p{z}\|_2^2 + \dfrac{\tau}{2}\|z_t\|_2^2 + \dfrac{c^2}{2b}\|\gamma^{1/2}u_t\|_2^2\end{align}(where we have ommited the variable $x$ in $\gamma(x)$) and \begin{align}\label{E0} E_0(t)& := \dfrac{1}{2}\|\alpha^{1/2}u_t\|_2^2 + \dfrac{c^2}{2}\|\p{u}\|_2^2\end{align} where we have ommited $x$ in $\alpha(x)$.

The next lemma guarantees that stability of solutions in $\mathbb{H}$ is equivalent to uniform exponential decay of the function $t \mapsto E(t).$

\begin{lemma}\label{equinorm}
	Let $\Phi = (u,u_t,u_{tt})$ be a weak solution for \eqref{usist} in $\mathbb{H}$ and assume condition \eqref{comp} holds. Then the following statements are equivalent: \begin{itemize}
		\item[\bf a)] $t \mapsto \|\Phi(t)\|_{\mathbb{H}}^2$ decays exponentially.
		
		\item[\bf b)] $t \mapsto \|M\Phi(t)\|_{\mathbb{H}}^2 = \|(u,z,z_t)\|_{\mathbb{H}}^2$ decays exponentially.
		
		\item[\bf c)] $t \mapsto E(t)$ decays exponentially.
	\end{itemize}
\end{lemma} \begin{proof}
Observe that $\Sigma: \mathbb{H} \to \mathbb{R}_+$ defined as \begin{equation*}
\Sigma((\xi_1,\xi_2,\xi_3)^\top) = \dfrac{c^2}{2}\|\p{\xi_1}\|_2^2+\dfrac{1}{2}\left\|\alpha^{1/2}\left(\xi_2 -\dfrac{c^2}{b}\xi_1\right)\right\|_2^2 + \dfrac{c^2}{2b}\left\|\gamma^{1/2}\left(\xi_2 - \dfrac{c^2}{b}\xi_1\right)\right\|_2^2 +  \dfrac{b}{2} \|\p{\xi_2}\|_2^2 + \dfrac{\tau}{2}\|\xi_3\|_2^2
\end{equation*} is such that $\|\cdot\|_{\mathbb{H}}^2 \sim \Sigma(\cdot).$ The proof follows by noticing that $E(t) = \Sigma (M\Phi(t)).$
\end{proof}

\begin{rmk}The equivalence in Lemma \ref{equinorm} is uniform with respect to $\gamma \in C(\overline{\Omega}).$\end{rmk}

The next proposition provides the set of main identities for the linear stabilization in $\mathbb{H}.$

\begin{prop}\label{id}
	If $(u,z,z_t)$ is a classical solution of \eqref{usistz} then for all $0 \leqslant s < t \leqslant T$ the following identities hold \begin{equation}\label{e1id}
	E_1(t)  + b\eta\int_s^t\|z_t\|_{\Gamma_0}^2d\sigma + \int_s^t\|\gamma^{1/2}u_{tt}\|_2^2d\sigma = E_1(s)+ \dfrac{c^2}{b}\int_s^t( f,\gamma u_{tt})d\sigma
	\end{equation} \begin{equation}
	\label{zmul} \int_s^t\left[b\|\p{z}\|_2^2 -\|z_t\|_2^2\right]d\sigma = - \int_s^t(\gamma u_{tt},z)d\sigma - \left[(z_t,z)+\dfrac{b\eta}{2}\|z\|_{\Gamma_0}^2\right]\biggr\rvert_s^t + \int_s^t(f,z)d\sigma
	\end{equation}
	\begin{align}\int_s^t \left[ \dfrac{n}{2} \|z_t\|_2^2 - \dfrac{b(n-2)}{2}\|\p{ z}\|_2^2 \right] d\sigma &= (b+1)\int_s^t\int_{\Gamma_0}\left|z_t\right|^2h\nu d\Gamma_0d\sigma - \int_s^t(\gamma u_{tt},h\nabla{z})d\sigma \nonumber \\ &- (z_t,h\nabla{z})\biggr\rvert_s^t + \int_s^t(f,h\nabla{z})d\sigma \label{hzmult}\end{align}  where $h$ is the vector field defined as $h(x) = x-x_0$, see Assumption \ref{asgeo}.
\end{prop}

\begin{proof} \begin{itemize} \item[\bf 1.] {\bf Proof of \eqref{e1id}:}
	Let, on $\mathbb{H}$, the bilinear form $\langle \cdot, \cdot \rangle$ be given by \begin{align}
	\langle (\xi_1,\xi_2,\xi_3)^\top, (\varphi_1,\varphi_2,\varphi_3) \rangle &= b\left(\p{\left(\xi_2 + \dfrac{c^2}{b}\xi_1\right)},\p{\left(\varphi_2 + \dfrac{c^2}{b}\varphi_1\right)}\right) \nonumber \\ &+ \left(\xi_3 + \dfrac{c^2}{b}\xi_2,\varphi_3 + \dfrac{c^2}{b}\varphi_2\right) + \dfrac{ c^2}{b}(\xi_2,\gamma\varphi_2),
	\end{align} which is clealy continuous. Moreover, recalling that $\Phi(t) = (u(t),u_t(t),u_{tt}(t))$ it follows that $2E_1(t) = \langle\Phi(t),\Phi(t)\rangle$ (see \eqref{E1}) therefore -- mostly ommiting the parameters $t \in [0,T]$ and $x$ in $\gamma(x)$ we obtain \begin{align}
	\dfrac{dE_1(t)}{dt} &= \left\langle \dfrac{d\Phi(t)}{dt},\Phi(t)\right\rangle = \langle \mathscr{A}\Phi(t)+F,\Phi(t)\rangle \nonumber \\ &= \langle (u_{t},u_{tt},-\alpha u_{tt} - c^2 A(u -\eta NN^\ast Au_t) - b A(u_t -\eta NN^\ast Au_{tt}))^\top,(u,u_t,u_{tt})\rangle + \dfrac{c^2}{b}(f,\gamma u_{tt}) \nonumber \\ & {\color{blue}\overset{\mathrm{(*)}}{=}} \ b\left(\p{z_t},\p{z}\right) + \left(-\gamma u_{tt} - c^2 A(u + NN^\ast Au_t) - b A(u_t + NN^\ast Au_{tt}),z_t\right) \nonumber \\ &+ \dfrac{c^2}{b}(\gamma u_{tt},u_t) + \dfrac{c^2}{b}( f, \gamma u_{tt}) = - \|\gamma^{1/2}u_{tt}\|_2^2 -b\eta\|z_t\|_{\Gamma_0}^2 + \dfrac{c^2}{b}( f,\gamma u_{tt}),
	\end{align} where, in $(*)$ above we have used the definition of $\langle \cdot, \cdot \rangle$ along with $bz(t) = u_t(t) + c^2 u(t)$ from what it follows that $-\alpha u_{tt}(t) + \frac{c^2}{b}u_{tt}(t) = -\gamma u_{tt}(t)$ and we directly computed \begin{align}
	\label{aux0} \left(bA(z + \eta NN^\ast Az_t),z_t\right) &= b(\p{(z+\eta NN^\ast Az_t)},\p{z_t}) \nonumber \\ &= b(\p{z},\p{z_t}) + b\eta\|N^\ast Az_t\|_{L^2(\Gamma)}^2 \nonumber \\ &= b(\p{z},\p{z_t}) + b\eta \|z_t\|_{L^2(\Gamma)}^2
	\end{align} by using self-adjointness of $A$ and the characterization \eqref{neq}. The identity \eqref{e1id} then follows by integration on $(s,t)$ for $0 \leqslant s < t \leqslant T.$ 

\item[\bf 2.] {\bf Proof of \eqref{zmul}:} Observe that taking the $L^2$--inner product of $$z_{tt} + bA(z + \eta ANN^\ast Az_t) = -\gamma u_{tt} + f$$ with $z$ we have, for the left hand side: \begin{align*}
(z_{tt} + bA(z + \eta ANN^\ast Az_t),z) &= \dfrac{d}{dt}(z_t,z) - \|z_t\|_2^2 + b\|\p{z}\|_2^2 + \dfrac{b\eta}{2}\dfrac{d}{dt}\|z_t\|_{\Gamma_0}^2,
\end{align*} while for the right hand side we have \begin{align*}
(-\gamma u_{tt} + f, z) = (-\gamma u_{tt}, z) + (f,z).
\end{align*}Then, putting right and left hand sides together and integrating on $(s,t)$ for $0 \leqslant s < t \leqslant T$ yelds 
\eqref{zmul}.

\item[\bf 3.] {\bf Proof of \eqref{hzmult}:} From divergence theorem, recall that for a vector field $h$ and a function $\varphi: \n{R}^n \to \n{R}$ we have that \begin{equation} \label{div} \io h\nabla \varphi d\Omega = \ig \varphi h\cdot \nu d\Gamma - \io \varphi\dv{h} d\Omega.
	\end{equation}

Considering $h(x) = x - x_0$, $x \in \overline{\Omega}$ we go back to the original (non-abstract) $z$--equation \begin{equation}\label{eqor}z_{tt} - b\Delta z = -\gamma u_{tt} + f\end{equation} -- with boundary conditions $\dfrac{\partial z}{\partial \nu} + \eta z_t = 0$ on $\Gamma_0$ and $z = 0$ on $\Gamma_1$ -- and multiply it by $h\nabla z$. We next analyze each of the involved terms in the resulting expression, ommiting the variable $t$ in most steps. For the first term we obtain $$(z_{tt},h\nabla z) = \dfrac{d}{dt}(z_t,h\nabla z) - (z_t,h\nabla z_t)$$ which, by chain rule $\nabla (\theta^2) = 2\theta \nabla \theta$ can be rewritten as \begin{equation}
\label{a11} (z_{tt},h\nabla z) = \dfrac{d}{dt}(z_t,h\nabla z)- \dfrac{1}{2} \io h\nabla(z_t^2(t))d\Omega\end{equation} and then we can apply the Divergence Theorem (with $\varphi = z_t^2$). For this we notice that since $h = x-x_0 = (x_1 - {x_0}_1, \cdots, x_n - {x_0}_n)$, we have$$\dv{h} = \sum\limits_{k=0}^n \dfrac{\partial (x_i-{x_0}_i)}{\partial x_i} = n.$$
Hence recalling that $z_t(t) = 0$ on $\Gamma_1$ for all $t$, we can further rewrite \eqref{a11} as \begin{align}
(z_{tt},h\nabla z) &= \dfrac{d}{dt} (z_t,h\nabla z) - \dfrac{1}{2} \int_{\Gamma_0} z_t^2h\nu d\Gamma_0 + \dfrac{n}{2} \|z_t\|_2^2. \label{a12}\end{align}

Moving to the next term, we first apply Green's first Theorem to get \begin{align}
(\Delta z,h\nabla z) &= - (\nabla z, \nabla(h\nabla z)) + \left(h\nabla z,\dfrac{\partial z}{\partial \nu}\right)_{\Gamma_0} \label{a13}
\end{align} and then recalling the product rule for gradients along with the fact that the Jacobian Matrix of $h$ is the identity we have \begin{align*}
\nabla z\nabla (h\nabla z) &= \nabla z(h\nabla(\nabla z) + Jh\nabla z) \nonumber \\&= h\nabla z\nabla(\nabla z) + Jh|\nabla z|^2 = \dfrac{h}{2} \nabla(|\nabla z|^2) + |\nabla z|^2
\end{align*} which then allows us to rewrite \eqref{a13} as \begin{align}
(\Delta z,h\nabla z) &= - \dfrac{1}{2}\io h\nabla(|\nabla z|^2)d\Omega - \|\nabla z\|_2^2 + \left(h\nabla z,\dfrac{\partial z}{\partial \nu}\right)_{\Gamma_0}, \label{a14}
\end{align} and then again application of divergence theorem (with $\varphi = |\nabla z|^2$) gives \begin{align}
(\Delta z,h\nabla z) &= - \dfrac{1}{2}\int_{\Gamma_0} |\nabla z|^2h\nu d\Gamma_0 + \left(\dfrac{n}{2}-1\right) \|\nabla z\|_2^2 + \left(h\nabla z,\dfrac{\partial z}{\partial \nu}\right)_{\Gamma_0}, \label{a15}
\end{align} and finally recalling the definition of normal derivative $\left(\dfrac{\partial z}{\partial \nu} = \nabla z \cdot \nu\right)$ we get \begin{align}
(\Delta z,h\nabla z) &= \dfrac{1}{2}\int_{\Gamma_0} |\nabla z|^2h\nu d\Gamma_0 + \left(\dfrac{n}{2}-1\right)\|\nabla z\|_2^2, \label{a16}. \end{align} 

The identity \eqref{hzmult} then follows by putting together equations \eqref{a16} (multiplied by $-b$) and \eqref{a12} along with the right hand side of \eqref{eqor} multiplied by $h\nabla z$.
\end{itemize}
\end{proof}

\begin{rmk}
	Take $f = 0.$ Observe that if $\gamma \equiv 0$ and one has no other source of dissipation -- zero boundary data and no interior damping -- then it follows from identity \eqref{e1id} that the $E_1(t) \equiv E_1(0),$ for all $t \geqslant 0.$ With the presence of the boundary dissipation, however, $E_1$ is decreasing even for $\gamma \equiv 0.$
\end{rmk}

\begin{theorem}\label{tdatko}
	Let $\Psi = (u,z,z_t)$ be a classical solution of \eqref{usistz} in $\mathbb{H}$. Then for all $0 \leqslant s < t \leqslant T$, if $f = 0$, the following estimate holds \begin{align}\label{stab1}
	E(t) + \int_s^t E(\sigma)d\sigma \lesssim E(s).
	\end{align}
\end{theorem} \begin{proof} \begin{itemize}
	\item[\bf Step 1.] Take $f=0$ and let $\varepsilon > 0$ to be given. H$\ddot{\mbox{o}}$lder's Inequality\footnote{H$\ddot{\mbox{o}}$lder's Inequality: $\|fg\|_1 \leqslant \|f\|_p\|g\|_q$ for $p,q \in [1,\infty)$ such that $p + q = pq.$ Here we used for $p = q = 2.$} along with Trace Theorem\footnote{Trace Theorem: $\|f\|_{\Gamma}^2 \lesssim \|\p{f}\|_2^2.$} and identity \eqref{e1id} allow the left hand side of identity \eqref{zmul} to be estimated as \begin{align}
\label{zmulin1} \int_s^t\left[b\|\p{z}\|_2^2 -\|z_t\|_2^2\right]d\sigma \lesssim E_1(s) +  \varepsilon \overline{\gamma} \int_s^t\|\p{z}\|_2^2d\sigma,
\end{align} where $\overline{\gamma} = \sup\limits_{x \in \overline{\Omega}} \gamma(x).$ In fact, we estimate the terms on the right side of \eqref{zmul} as follows: \begin{align}
- \int_s^t(\gamma u_{tt},z)d\sigma - \left[(z_t,z)+\dfrac{b}{2}\|z\|_{\Gamma_0}^2\right]\biggr\rvert_s^t & \lesssim  C_\varepsilon \int_s^t \|\gamma^{1/2}u_{tt}\|_2^2 + \varepsilon \overline{\gamma} \int_s^t \|\p{z}\|_2^2 \nonumber \\ &+ E_1(t) + E_1(s)
\end{align} and then, under \eqref{e1id} with $f = 0$, \eqref{zmulin1} follows.\item[\bf Step 2.] Next, recalling that $\max\limits_{x \in \overline{\Omega}} |h(x)| < \infty$ and using again identity \eqref{e1id} we also estimate the left hand side of \eqref{hzmult} as \begin{align}\int_s^t \left[ \dfrac{n}{2} \|z_t\|_2^2 - \dfrac{b(n-2)}{2}\|\p{ z}\|_2^2 \right] d\sigma &\lesssim  E_1(s) +  \varepsilon \overline{\gamma} \int_s^t\|\p{z}\|_2^2d\sigma, \label{hzmultin1}\end{align} for which we have used the fact that, for $f = 0$, identity \eqref{e1id} allows us to control the time integrals $\displaystyle\int_s^t\|z_t\|_{\Gamma_0}^2d\sigma$ and $\displaystyle\int_s^t\|\gamma^{1/2}u_{tt}\|_2^2 d\sigma$  along with 
$$(\gamma u_{tt} , h\nabla z ) \lesssim \epsilon \bar{\gamma}||\nabla z||^2 + C_{\epsilon} ||\gamma^{1/2} u_{tt}||^2 $$

\item[\bf Step 3.] Now notice that adding \eqref{hzmultin1} with $(n-1)/2$ -- times \eqref{zmulin1} gives \begin{align*}
 \int_s^t\left[\dfrac{b}{2}\|\p{z}\|_2^2 +\dfrac{1}{2}\|z_t\|_2^2\right]d\sigma \lesssim E_1(s) +  \varepsilon \overline{\gamma} \int_s^t\|\p{z}\|_2^2d\sigma,
\end{align*} from where it follows, by taking $\varepsilon$ small, that \begin{align}\label{zin}
\int_s^t\left[\|\p{z}\|_2^2 +\|z_t\|_2^2\right]d\sigma \lesssim E_1(s).
\end{align}

\begin{rmk}
	Notice that none of the arguments for Steps 1-3 depend on the requirement that $\gamma > 0.$ Therefore it is valid for $\gamma \geqslant 0.$
\end{rmk}

\item[\bf Step 4.] Finally, from $bz = bu_t + c^2u$ we get, from \eqref{zin} \begin{align}
\label{ueq} \|\p{u}\|_2^2 + \int_s^t\|\p{u}\|_2^2d\sigma \lesssim \int_s^t \|\p{z}\|_2^2d\sigma \lesssim E_1(s),
\end{align} and then \eqref{stab1} follows by adding \eqref{ueq}, \eqref{zin} and \eqref{e1id}. \end{itemize}
This proves the inequality in Theorem \ref{tdatko} valid for classical solutions. The extension to mild solutions follows from the density of the domain of the generator in $\mathbb{H}$ and from 
weak lower-semi-continuity of the energy functions. 
\end{proof}
The proof of the final result in Theorem \ref{t2} follows from Datko's Theorem \cite{pazy_semigroups_1992}. 


		\bibliographystyle{acm} 
		\bibliography{stab_ref.bib}
\end{document}